\documentclass{amsart} 
\newtheorem{Thm}{Theorem}[section]
\newtheorem{Prop}[Thm]{Proposition}

\newtheorem{Lem}[Thm]{Lemma}
\numberwithin{equation}{section}
\begin{document} 

\title[$p$-Royden and $p$-harmonic boundaries of a graph]{Some results concerning the $p$-Royden and $p$-harmonic boundaries of a graph of bounded degree}

\author[M. J. Puls]{Michael J. Puls}
\address{Department of Mathematics \\
John Jay College-CUNY \\
445 West 59th Street \\
New York, NY 10019 \\
USA}
\email{mpuls@jjay.cuny.edu}

\begin{abstract}
Let $p$ be a real number greater than one and let $\Gamma$ be a connected graph of bounded degree. We show that the $p$-Royden boundary of $\Gamma$ with the $p$-harmonic boundary removed is an $F_{\sigma}$-set. We also characterize the $p$-harmonic boundary of $\Gamma$ in terms of the intersection of the extreme points of a certain subset of one-sided infinite paths in $\Gamma$. 
\end{abstract}


\date{July 7, 2011}
\maketitle

\section{Introduction}\label{Introduction}
Let $\Gamma$ be a graph with vertex set $V_{\Gamma}$ and edge set $E_{\Gamma}$. We will write $V$ for $V_{\Gamma}$ and $E$ for $E_{\Gamma}$ if it is clear what graph $\Gamma$ we are working with. For $x \in V, \mbox{ deg}(x)$ will denote the number of neighbors of $x$ and $N_x$ will be the set of neighbors of $x$. A graph $\Gamma$ is said to be of {\em bounded degree} if there exists a positive integer $k$ such that $\mbox{deg}(x) \leq k$ for every $x \in V$. A path $\gamma$ in $\Gamma$ is a sequence of vertices $x_1, x_2, \dots, x_n$ where $x_{i+1} \in N_{x_i}$ for $1 \leq i \leq n-1$ and $x_i \neq x_j$ if $i \neq j$. Assume throughout this paper that all infinite paths have no self-intersections. A graph is connected if any two given vertices of the graph are joined by a path. All graphs considered in this paper will be connected, of bounded degree with no self-loops and have countably infinite number of vertices. We shall say that a subset $S$ of $V$ is connected if the subgraph of $\Gamma$ induced by $S$ is connected. The Cayley graph of a finitely generated group is an example of the type of graph that we study in this paper. By assigning length one to each edge of $\Gamma$, $V$ becomes a metric space with respect to the shortest path metric. We will denote this metric by $d( x, y)$, where $x$ and $y$ are vertices of $\Gamma$. Thus $d(x, y)$ gives the length of the shortest path joining the vertices $x$ and $y$. Finally, if $x \in V$ and $n \in \mathbb{N}$, then $B_n(x)$ will denote the metric ball that contains all elements of $V$ that have distance less than $n$ from $x$.

Let $p$ be a real number greater than one. In Section \ref{definepharm} we will define the $p$-Royden boundary of $\Gamma$, which we will indicate by $R_p(\Gamma)$. We will also define the $p$-harmonic boundary of $\Gamma$, which is a subset of $R_p(\Gamma)$. We will use $\partial_p(\Gamma)$ to denote the $p$-harmonic boundary. Our motivation for investigating the $p$-harmonic boundary of a graph is its connection to the vanishing of the first reduced $\ell^p$-cohomology space of a finitely generated group. More specifically, this space vanishes if and only if the $p$-harmonic boundary of the group is empty or contains exactly one element, see \cite[Section 7]{PulsPMJ} for the details of this fact. Gromov conjectured in \cite[page 150]{Gromov} that the first reduced $\ell^p$-cohomology  space of a finitely generated amenable group vanishes. Thus, a better understanding of the $p$-harmonic boundary could be helpful in resolving Gromov's conjecture. 

Recall that in a topological space a set is said to be $F_{\sigma}$ if it is a countable union of closed sets. In this paper we will prove that $R_p(\Gamma) \setminus \partial_p(\Gamma)$ is $F_{\sigma}$. For each infinite path in $\Gamma$ we can associate a set of extreme points, which is roughly the ``points at infinity'' of the path with respect to the $p$-Royden boundary. Our other main result in this paper is that the $p$-harmonic boundary is precisely the intersection of the extreme points of a certain subset of one-sided infinite paths in $\Gamma$.

The research for this paper was partially supported by PSC-CUNY grant 63873-00 41 and I would like to thank them for their support.

\section{The $p$-Royden and $p$-harmonic boundaries}\label{definepharm}
Let $1 < p \in \mathbb{R}$. In this section we construct the $p$-Royden and $p$-harmonic boundaries of $\Gamma$. For a more detailed discussion about this construction see Section 2.1 of \cite{PulsPMJ}. Before we can give these definitions we need to define the space of $p$-Dirichlet finite functions on $V$. For any $S \subset V$, the outer boundary $\partial S$ of $S$ is the set of vertices in $V\setminus S$ with at least one neighbor in $S$. For a real-valued function $f$ on $S \cup \partial S$ we define the $p$-th power of the {\em gradient}, the {\em $p$-Dirichlet sum}, and the {\em $p$-Laplacian} of $x \in S$ by
\begin{equation*}
\begin{split}
 \vert Df (x) \vert^p  & = \sum_{y \in N_x} \vert f(y) - f(x) \vert^p,   \\
 I_p (f, S)            & = \sum_{x \in S} \vert Df (x) \vert^p, \\
  \Delta_p f (x)    & = \sum_{y \in N_x} \vert f(y) - f(x) \vert^{p-2} (f(y) - f(x)).
\end{split}
\end{equation*} 
In the case $1 < p < 2$, we make the convention that $\vert f(y) - f(x) \vert^{p-2} (f(y) - f(x)) = 0$ if $f(y) = f(x)$. Let $S \subseteq V$. A function $f$ is said to be $p$-harmonic on $S$ if $\Delta_p f(x) = 0$ for all $x \in S$. We shall say that $f$ is {\em $p$-Dirichlet finite} if $I_p(f,V) < \infty$. The set of all $p$-Dirichlet finite functions on $\Gamma$ will be denoted by $D_p(\Gamma)$. With respect to the following norm $D_p(\Gamma)$ is a reflexive Banach space,
$$ \parallel f \parallel_{D_p} = \left( I_p(f,V) + \vert f(o) \vert^p \right)^{1/p},$$
where $o$ is a fixed vertex of $\Gamma$ and $f \in D_p(\Gamma)$. We use $HD_p(\Gamma)$ to represent the set of $p$-harmonic functions on $V$ that are contained in $D_p(\Gamma)$. Let $\ell^{\infty}(\Gamma)$ denote the set of bounded functions on $V$ and let $\parallel f \parallel_{\infty} = \sup_V \vert f \vert$ for $f \in \ell^{\infty}(\Gamma)$. Set $BD_p(\Gamma) = D_p(\Gamma) \cap \ell^{\infty}(\Gamma)$. The set $BD_p(\Gamma)$ is a Banach space under the norm 
$$\parallel f \parallel_{BD_p} = \left( I_p(f, V)\right)^{1/p} + \parallel f \parallel_{\infty},$$
where $f \in BD_p(\Gamma)$. Let $BHD_p(\Gamma)$ be the set of bounded $p$-harmonic functions contained in $D_p(\Gamma)$. The space $BD_p(\Gamma)$ is also closed under the usual operations of scalar multiplication, addition and pointwise multiplication. Furthermore, $\parallel fg \parallel_{BD_p} \leq \parallel f \parallel_{BD_p} \parallel g \parallel_{BD_p}$ for $f, g \in BD_p(\Gamma)$. Thus $BD_p(\Gamma)$ is a commutative Banach algebra. Let $C_c(\Gamma)$ be the set of functions on $V$ with finite support. Indicate the closure of $C_c(\Gamma)$ in $D_p(\Gamma)$ by $\overline{C_c(\Gamma)}_{D_p}$. Set $B(\overline{C_c(\Gamma)}_{D_p}) = \overline{C_c(\Gamma)}_{D_p} \cap \ell^{\infty}(\Gamma)$. Using the fact that the inequality $(a+b)^{1/p} \leq a^{1/p} + b^{1/p}$ is true when $a,b \geq 0$ and $1 < p \in \mathbb{R}$, we see immediately that $\parallel f \parallel_{D_p} \leq \parallel f \parallel_{BD_p}$. It now follows that  $B(\overline{C_c(\Gamma)}_{D_p})$ is closed in $BD_p(\Gamma)$. 

Let $Sp(BD_p(\Gamma))$ denote the set of complex-valued characters on $BD_p(\Gamma)$, that is the nonzero ring homomorphisms from $BD_p(\Gamma)$ to $\mathbb{C}$. Then with respect to the weak $\ast$-topology, $Sp(BD_p(\Gamma))$ is a compact Hausdorff space. Given a topological space $X$, let $C(X)$ denote the ring of continuous functions on $X$ endowed with the sup-norm. The Gelfand transform defined by $\hat{f}(\chi) = \chi(f)$ yields a monomorphism of Banach algebras from $BD_p(\Gamma)$ into $C(Sp(BD_p(\Gamma)))$ with dense image. Furthermore the map $i \colon V \rightarrow Sp(BD_p(\Gamma))$ given by $(i(x))(f) = f(x)$ is an injection, and $i(V)$ is an open dense subset of $Sp(BD_p(\Gamma))$. For the rest of this paper, we shall write $f$ for $\hat{f}$, where $f \in BD_p(\Gamma)$. The {\em $p$-Royden boundary} of $\Gamma$, which we shall denote by $R_p(\Gamma)$, is the compact set $Sp(BD_p(\Gamma))\setminus i(V)$. The {\em $p$-harmonic boundary} of $\Gamma$ is the following subset of $R_p(\Gamma)$:
\[  \partial_p(\Gamma) \colon = \{ \chi \in R_p(\Gamma) \mid \hat{f}(\chi) = 0 \mbox{ for all }f \in B(\overline{C_c(\Gamma)}_{D_p}) \}. \]

Let $S$ be an infinite subset of $V$ and let $A$ and $B$ be disjoint nonempty subsets of $S \cup \partial S$. The {\em $p$-capacity} of the condenser $(A, B, S)$ is defined by 
\[ cap_p(A, B, S) = \inf_u I_p(u), \]
where the infimum is taken over all functions $u \in D_p(\Gamma)$ with $u = 0$ on $A$ and $u=1$ on $B$. Such a function is called {\em admissible}. Set $cap_p(A, B, S) = \infty$ if the set of admissible functions is empty.

Let $A$ be a finite subset of $S \cup \partial S$ and let $(U_n)$ be an exhaustion of $V$ by finite connected subsets such that $A \subset U_1$. We now define
\[ cap_p(A, \infty, S) = \lim_{n \rightarrow \infty} cap_p(A, (\partial S \cup S) \setminus U_n, S). \]
Since $cap_p(A, (\partial S \cup S)\setminus U_n, S) \geq cap_p(A, (\partial S \cup S)\setminus U_{n+1}, S)$, the above limit exists. We shall say that $S$ is {\em $p$-hyperbolic} if there exists a finite subset $A$ of $S \cup \partial S$ that satisfies $cap_p(A, \infty, S) > 0.$ If $S$ is not $p$-hyperbolic, then it is said to be { \em $p$-parabolic}. An equivalent definition of $p$-parabolic is that $S$ is $p$-parabolic if and only if $1_S \in \overline{C_c(\Gamma_S)}_{D_p}$, where $1_S$ is the constant function 1 on $S$ and $\Gamma_S$ the subgraph of $\Gamma$ induced by $S$, \cite[Theorem 3.1]{Yam76}. We will define a graph $\Gamma$ to be $p$-hyperbolic ($p$-parabolic) if its vertex set $V$ is $p$-hyperbolic ($p$-parabolic). It was shown in \cite[Proposition 4.2]{PulsPMJ} that $\Gamma$ is $p$-parabolic if and only if $\partial_p(\Gamma) = \emptyset$. A useful property of $p$-hyperbolic graphs that we will use throughout this paper is the following $p$-Royden decomposition, see \cite[Theorem 4.6]{PulsPMJ} for a proof.
\begin{Thm} \label{pRoydendecomp} ($p$-Royden decomposition) Let $1 < p \in \mathbb{R}$ and suppose $f \in BD_p(\Gamma)$. Then there exists a unique $u \in B(\overline{ C_c( \Gamma)}_{D_p})$ and a unique $h \in BHD_p(\Gamma)$ such that $f = u +h$.  
\end{Thm}

Let $G$ be a finitely generated group. The Cayley graph of $G$ is an example of the type of graph we study in this paper. As was mentioned in Section \ref{Introduction} the first reduced $\ell^p$-cohomology space of $G$ vanishes if and only if the cardinality of the $p$-harmonic boundary of $G$ is one or zero. The reason for this is that the first reduced $\ell^p$-cohomology space of $G$ vanishes if and only if the only $p$-harmonic functions on $G$ that are contained in $D_p(G)$ are the constants, for a proof of this see the remark after Theorem 3.5 in \cite{Puls06}. Furthermore, \cite[Theorem 2.5]{PulsPMJ} tells us that there are nonconstant $p$-harmonic functions with finite $p$-Diriclet sum on a graph of bounded degree if and only if the cardinality of the $p$-harmonic boundary of the graph is greater than one. In section 7 of \cite{PulsPMJ} the $p$-harmonic boundary is computed for several groups.

\section{Statement of main results}\label{statementresults}
In this section we will state our main results. In section \ref{prooffsigma} we will prove
\begin{Thm}\label{fsigma}
Let $1 < p \in \mathbb{R}$ and let $\Gamma$ be a graph of bounded degree. The set $R_p(\Gamma) \setminus \partial_p(\Gamma)$ is $F_{\sigma}$.
\end{Thm}

Before we state our other main result we need to define the set of extreme points of a path in $\Gamma$. Let $P$ be the set of all one-sided infinite paths with no self-intersections in $\Gamma$. For a real-valued function $f$ on $V$ and a path $\gamma \in P$, the limit of $f$ as we follow $\gamma$ to infinity is given by $\lim_{n \rightarrow \infty} f(x_n)$, where $x_0,x_1, \dots, x_n, \dots $ is the vertex representation of the path $\gamma$. Sometimes we write $f(\gamma) = \lim_{n \rightarrow \infty} f(x_n)$ to indicate this limit. Let $\gamma \in P$ and denote by $V(\gamma)$ the set of vertices on $\gamma$. The closure of $i(V(\gamma))$ in $Sp(BD_p(\Gamma))$ will be indicated by $\overline{V}(\gamma)$. Recall that $Sp(BD_p(\Gamma))$ is endowed with the $\mbox{weak}*$-topology. Thus $\chi \in \overline{V}(\gamma)$ if and only if there exists a subsequence $(x_{n_k})$ of $(x_n)$ such that $\lim_{k \rightarrow \infty} f(x_{n_k}) = \chi(f)$ for all $f \in BD_p(\Gamma)$. The {\em extreme points} of a path $\gamma$ is defined to be 
\[ E(\gamma) = \overline{V}(\gamma) \cap R_p(\Gamma). \]
Let $f \in B(\overline{C_c(\Gamma)}_{D_p})$ and set $A_f = \{\gamma \in P \mid f(\gamma) \neq 0\}$. Set 
\[ E_f = \overline{ \{ \cup_{\gamma} E(\gamma) \mid \gamma \in P\setminus A_f \}}. \]
In Section \ref{proofboundchar} we shall prove
\begin{Thm} \label{boundchar}
Let $1 < p \in \mathbb{R}$ and let $\Gamma$ be a graph of bounded degree. Then 
$$ \partial_p(\Gamma) = \bigcap_{f \in B(\overline{C_c(\Gamma)}_{D_p})} E_f .$$
\end{Thm}

Let $1 < p \in \mathbb{R}$. If $\Gamma$ is $p$-parabolic, then $\partial_p(\Gamma) = \emptyset$ and Theorem \ref{fsigma} is true. Also for the $p$-parabolic case, $1_V \in B(\overline{C_c(\Gamma)}_{D_p})$ by \cite[Theorem 3.2]{Yam76}, where $1_V$ is the constant function one on $V$. Then $E_{1_V} = \emptyset$ and Theorem \ref{boundchar} follows. Thus for the rest of the paper we will assume $\Gamma$ is $p$-hyperbolic.

\section{Proof of Theorem \ref{fsigma}} \label{prooffsigma}
In this section we will prove Theorem \ref{fsigma}. We will start by giving some needed definitions and proving a comparison principle. A comparison principle for finite subsets of $V$ was proved in \cite[Theorem 3.14]{HoloSoar}. Our proof follows theirs in spirit.

Let $f$ and $h$ be elements of $BD_p(\Gamma)$ and let $1 < p \in \mathbb{R}$. Define
\[ \langle \Delta_p h, f \rangle \colon = \sum_{x \in V} \sum_{y \in N_x} \vert h(y) - h(x) \vert^{p-2} (h(y)-h(x)) (f(y) - f(x)). \]
The sum exists since 
   \[ \sum_{x \in V} \sum_{y \in N_x} \left| \vert h(y) - h(x) \vert^{p-2} (h(y) - h(x)) \right|^q = I_p(h, V) < \infty, \]
where $\frac{1}{p} + \frac{1}{q} = 1$. For notational convenience let 
\[ T(h, f, x, y) = \vert h(y) - h(x) \vert^{p-2} (h(y) - h(x))(f(y) - f(x)). \]
In order to prove Theorem \ref{fsigma} we will need the following:
\begin{Lem} \label{comparisonprinciple} 
(Comparison principle) Let $h_1, h_2$ be elements of $BHD_p(\Gamma)$ and suppose $h_1(x) \leq h_2(x)$ for all $x \in \partial_p(\Gamma)$. Then $h_1 \leq h_2$ on $V$.
\end{Lem}
\begin{proof}Define a function $f$ on $V$ by $f = \min\{ h_2 - h_1, 0 \}$. Theorem 4.8 of \cite{PulsPMJ} says $f \in B(\overline{C_c(\Gamma)}_{D_p})$ since $f = 0$ on $\partial_p(\Gamma)$. By Lemma 4.6 of \cite{PulsPMJ} we have $\langle \Delta_p h_1, f \rangle = 0$ and $\langle \Delta_p h_2, f \rangle = 0$, which implies $\langle \Delta_p h_1 - \Delta_p h_2, f \rangle = 0$. Now set 
\begin{eqnarray*} 
A  &  =   &  \{ x \in V \mid h_1(x) \leq h_2(x) \} , \\
B  &  =   & \{  x \in V \mid h_2(x) < h_1(x) \}, 
\end{eqnarray*}
and for $a \in V$ let
\begin{eqnarray*}
C_a  &  =  & \{ y \in V \mid y \in N_a \mbox{ and } h_1(y) \leq h_2(y) \}, \\
D_a  &  =  &  \{ y \in V \mid y \in N_a \mbox{ and } h_2(y) < h_1(y) \}.
\end{eqnarray*}
Now
\begin{equation}
0 = \sum_{x \in V} \sum_{y \in N_x} ( T(h_1, f, x, y) - T(h_2, f, x, y)) = T_1 + T_2 + T_3 \label{eq:innerprodzero}
\end{equation}
where
\begin{eqnarray*}
T_1  &  =  &  \sum_{x \in A} \sum_{y \in C_x} ( T(h_1, f, x, y) - T(h_2, f, x, y)), \\
T_2 &  =   & \left( \sum_{x \in A} \sum_{y \in D_x} + \sum_{x \in B} \sum_{y \in C_x}\right)(T(h_1, f, x, y) - T(h_2, f, x, y)),
\end{eqnarray*}
and
\begin{eqnarray*}
T_3 & = &\sum_{x \in B} \sum_{y \in D_x} ( T(h_1, f, x, y) - T(h_2, f, x, y)).
\end{eqnarray*}
Since $f(x) = f(y) =0$ for $x \in A$ and $y \in C_x$ it follows that $T_1 = 0$. We now claim that $T_3  \leq 0$. To see the claim let $a$ and $b$ be real numbers such that $a \neq b$. It follows from the inequality
\[ \vert a \vert^{p-2} a (a-b) > \vert b \vert^{p-2} b (a - b) \]
that
\begin{equation}
 T(h_1, h_1-h_2, x, y) \geq T(h_2, h_1-h_2, x, y ). \label{eq:pharmineq}
\end{equation}
Equality occurs if and only if $(h_1-h_2)(x) = (h_1 - h_2)(y).$ Now if $x \in B$ and $y \in D_x$, then $f(y) - f(x) = (h_2-h_1)(y) - (h_2-h_1)(x)$. Combining (\ref{eq:pharmineq}) with the fact $T(h_k, h_1-h_2, x, y) = -T(h_k, h_2-h_1, x, y),$ where $k=1$ or $k=2$, we obtain $T_3 \leq 0$, which is our claim. 

We now proceed to show that if there is a pair of vertices $x$ and $y$ that satisfy $x\in A, y\in D_x$ or $x\in B, y\in C_x,$ then $T_2 < 0.$ Suppose $x \in A$ and $y \in D_x$. Then $f(y) - f(x) = h_2(y) - h_1(y) < 0$ and

\[ T(h_1, f, x, y) - T(h_2, f, x, y) = \left(h_2(y) - h_1(y)\right) \times  \]
\[ \left( \vert h_1(y) - h_1(x) \vert^{p-2} (h_1(y) - h_1(x)) - \vert h_2(y) - h_2(x) \vert^{p-2}(h_2(y) - h_2(x))\right). \]

Also $h_1(y) - h_1(x) > h_2(y) - h_2(x)$ because $h_1(y) - h_2(y) > 0 \geq h_1(x) - h_2(x).$ So if $h_2(y) \geq h_2(x)$ we see that $T(h_1, f, x, y) - T(h_2, f, x, y) < 0$ since $h_1(y) - h_1(x) > h_2(y) - h_2(x).$ On the other hand if $h_2(y) < h_2(x)$ and $h_1(y) > h_1(x)$ we obtain 
\[ T(h_1, f, x, y) - T(h_2, f, x, y) \]
\[ =(h_2(y) - h_1(y))\left( \vert h_1(y) - h_1(x) \vert^{p-1} + \vert h_2(y) - h_2(x) \vert^{p-1} \right) < 0 \]
since $\vert h_2(y) - h_2(x) \vert = - (h_2(y) - h_2(x)).$ The only other possibility is $h_2(y) < h_2(x)$ and $h_1(y) \leq h_1(x)$. If this is the case then $h_2(y) < h_1(y) \leq h_1(x) \leq h_2(x)$ due to $x \in A$ and $y \in D_x$. Consequently, $h_2(y) - h_2(x) < h_1(y) -h_1(x)$ and $h_1(x) - h_1(y) < h_2(x) - h_2(y)$; hence, $\vert h_1(y) - h_1(x) \vert < \vert h_2(y) - h_2(x) \vert.$ It now follows that
\[ T(h_1, f, x, y) - T(h_2, f, x, y) \]
\[ = (h_2(y) - h_1(y)) \left( \vert h_2(y) - h_2(x) \vert^{p-1} - \vert h_1(y) - h_1(x) \vert^{p-1} \right) < 0. \]
A similar argument can be used to show that $T(h_1, f, x, y) - T(h_2, f, x, y) < 0$ for each $x\in B$ and $y \in C_x.$ Hence, if $x \in A, y \in D_x$ or $x \in B, y \in C_x$, then $T_2 <0.$ Since $T_1 = 0$ and $T_3 \leq 0$, it follows from (\ref{eq:innerprodzero}) that it must be the case $T_2 = 0$. Thus it is impossible to have a pair of vertices $x$ and $y$ with $x \in A, y \in D_x$ or $x\in B, y\in C_x.$

Now assume that $h_1(z) > h_2(z)$ for some $z \in V$. We claim that there exists vertices $x_0, y_0$ in $V$ for which $y_0 \in N_{x_0}, h_1(x_0) > h_2(x_0)$ and $h_1( y_0) \leq h_2(y_0).$ To see the claim suppose $h_1 = h_2$ on $\partial_p(\Gamma)$, then $h_1 = h_2$ on $V$ by \cite[Corollary 4.9]{PulsPMJ}. So there exists an $x \in \partial_p(\Gamma)$ that satisfies $h_1(x) < h_2(x).$ Let $(x_n) \rightarrow x$ where $(x_n)$ is a sequence in $V$. Now there exists a term $x_m$ in this sequence such that $h_1(x_m) < h_2(x_m)$. Since $\Gamma$ is connected there is a path from $z$ to $x_m$. Thus there are vertices $x_0$ and $y_0$ on this path with $y_0 \in N_{x_0}, h_1(x_0) > h_2(x_0)$, and $h_1(y_0) < h_2(y_0)$ because $h_1(z) > h_2(z)$ and $h_1(x_m) < h_2(x_m)$. Thus $x_0 \in B$ and $y_0 \in C_{x_0},$ a contradiction. Therefore, $h_1(z) \leq h_2(z)$ for all $z \in V.$
\end{proof}

{\em Proof of Theorem \ref{fsigma}}. Let $1 < p \in \mathbb{R}$. Since $Sp(BD_p(\Gamma))$ is a normal space, there exists for each $x \in R_p(\Gamma)$ a sequence $(U_j(x))$ of open sets containing $x$ such that $\overline{U}_{j+1} (x) \subseteq U_j(x).$ For each $j \in \mathbb{N}$ there exists a finite number of points $x_{j,k}, 1 \leq k \leq N_j$ such that $U_j (x_{j,k})$ cover $R_p(\Gamma).$ For notational simplicity we will denote $U_j(x_{j,k})$ by $U_{j,k}.$ Using Urysohn's lemma we can construct a  continuous function $\phi_{j,k}$ with $\phi_{j,k} = 2$ on $U_{j,k}$ and $\phi_{j,k} =-1$ on $Sp(BD_p(\Gamma)) \setminus U_{j-1,k}$. By the density of $BD_p(\Gamma)$ in $C(Sp(BD_p(\Gamma)))$ there exists a $g \in BD_p(\Gamma)$ such that $\vert \phi_{j,k} - g \vert < \frac{1}{2}$. Set $f_{j,k} = \max(\min(1,g),0)$, so $f_{j,k} \in BD_p(\Gamma), 0 \leq f_{j,k} \leq 1, f_{j,k} = 1$ on $U_{j,k}$ and $f_{j,k} =0$ on $Sp(BD_p(\Gamma))\setminus U_{j-1,k}$.  The $p$-Royden decomposition of $BD_p(\Gamma)$ yields a unique $p$-harmonic function $h_{j,k}\in BHD(\Gamma)$ and a unique $u_{j,k} \in B(\overline{C_c(\Gamma)}_{D_p})$ such that $f_{j,k} = u_{j,k} + h_{j,k}$. Because $u_{j,k} = 0$ on $\partial_p(\Gamma)$ by \cite[Theorem 4.8]{PulsPMJ}, we see that $f_{j,k} = h_{j,k}$ on $\partial_p(\Gamma)$. Now define 
\[ R_{j,k} = \{ x \in R_p(\Gamma) \cap \overline{U}_{j,k} \mid \lim_{x_n \rightarrow x} h_{j,k}(x_n) < f_{j,k}(x) =1 \}, \]
where $(x_n)$ is a sequence in $V$. Observe that if $R_{j,k}$ is nonempty, then it only contains elements of $R_p(\Gamma)\setminus \partial_p(\Gamma).$

Let $x \in R_p(\Gamma) \setminus \partial_p(\Gamma).$ We will now show that there exists $j,k \in \mathbb{N}$ such that $x \in R_{j,k}$. Since $x \notin \partial_p(\Gamma)$ there exists a $u \in B(\overline{C_c(\Gamma)}_{D_p})$ such that $u(x) \neq 0$. Since $B(\overline{C_c(\Gamma)}_{D_p})$ is an ideal we may assume that $u \geq 0$ on $V$ and $u(x) >0$. Replacing $u$ by $u^{-1}(x) u$ if necessary we may assume that $u(x) = 1$. Let $h \in BHD_p(\Gamma)$ that satisfies $h \geq 1$ on $V$. Set $f = u + h$, so $f \in BD_p(\Gamma)$ and $f=h$ on $\partial_p(\Gamma)$. Let $(x_n)$ be a sequence in $V$ that converges to $x$. Now $\lim_{n \rightarrow \infty} h(x_n) < f(x)$. Because $f$ is continuous we can find an open set $U_{j,k}$ that contains $x$ and satisfies 
\[ m = \inf_{U_{j-1,k} \cap R_p(\Gamma)} f > \lim_{n \rightarrow \infty} h(x_n).\]
It now follows
\[ f_{j,k} \leq \frac{f}{m} \mbox{ on } R_p(\Gamma), \]
which implies that $h_{j,k} \leq \frac{h}{m}$ on $\partial_p (\Gamma)$. An appeal to the comparison principle gives us 
\[ \lim_{n \rightarrow \infty} h_{j,k} (x_n) \leq \frac{1}{m} \lim_{n \rightarrow \infty} h(x_n) < 1 = f_{j,k} (x), \]
hence $x \in R_{j,k}$. Furthermore,
$$ R_{j,k} = \bigcup_{i = 1}^{\infty} \left( R_p(\Gamma) \cap \overline{U}_{j,k} \cap \{ \overline{ y \in V \mid h_{j,k} (y) < 1 - 1/i} \} \right). $$
Thus $R_{j,k}$ is a countable union of compact sets. Theorem \ref{fsigma} now follows because 
$$ R_p(\Gamma) \setminus \partial_p(\Gamma) = \bigcup_{j = 1}^{\infty} \bigcup_{k =1}^{N_j} R_{j,k}. $$

\section{Proof of Theorem \ref{boundchar}} \label{proofboundchar}
Before we prove Theorem \ref{boundchar} we need to state some definitions and prove several preliminary results. 

Fix a real number $p >1$. Recall that $E$ denotes the edge set of a graph $\Gamma$. Denote by $\mathcal{F}(E)$ the set of all real-valued functions on $E$ and let $\mathcal{F}^+(E)$ be the subset of $\mathcal{F}(E)$ that consists of all nonnegative functions. For $\omega \in \mathcal{F}(E)$ set 
\[ \mathcal{E}_p (\omega) = \sum_{e \in E} \vert \omega (e) \vert^p. \]
The edge set of a path $\gamma$ in $\Gamma$ will be denoted by $Ed(\gamma)$, remember $E(\gamma)$ represents the extreme points of $\gamma$. Let $Q$ be a set of paths in $\Gamma$, denote by $\mathcal{A}(Q)$ the set of all $\omega \in \mathcal{F}^+(E)$ that satisfy $\mathcal{E}_p(\omega) < \infty$ and $\sum_{e \in Ed(\gamma)} \omega(e) \geq 1$ for each $\gamma \in Q$. The {\em extremal length} of order $p$ for $Q$ is defined by
\[ \lambda_p(Q)^{-1} = \inf\{\mathcal{E}_p(\omega) \mid \omega \in \mathcal{A}(Q) \}. \]
A variation of the next lemma was proved for the case $p=2$ in \cite[Lemma 6.13]{Soardibook}. In the $p=2$ case the conclusion of the lemma is stronger in that $g$ belongs to $\overline{C_c(\Gamma)}_{D_2}$ instead of the larger space $D_2(\Gamma)$. 

\begin{Lem} \label{gequalinfinity}
Let $K$ be a compact subset of $R_p(\Gamma)$ with $K \cap \partial_p(\Gamma) = \emptyset$. Then there exists a function $g \in D_p(\Gamma)$ that satisfies $g = \infty$ on $K$ and $g=0$ on $\partial_p(\Gamma)$.
\end{Lem}
\begin{proof}
By Urysohn's lemma there exists an $f \in C(Sp(BD_p(\Gamma)))$ that satisfies the following: $0 \leq f \leq 1, f=1$ on $K$ and $f = 0$ on $\partial_p(\Gamma)$. Using the argument from the first paragraph of the proof of Theorem \ref{fsigma} we may and do assume $f \in BD_p(\Gamma)$.

Let $(U_n)$ be an exhaustion of $V$ by finite connected subsets. Applying Theorem 3.1 of \cite{HoloSoar} yields a function $h_n$ that is $p$-harmonic on $U_n$ and equals $f$ on $V\setminus U_n$. It follows from the minimizer property of $p$-harmonic functions on $U_n$ that $I_p(h_n, V) \leq I_p(f,V)$. Hence, $h_n \in BD_p(\Gamma)$ for each $n \in \mathbb{N}$. Also, $h_n = 0$ on $\partial_p(\Gamma), h_n = 1$ on $K$ and $0 \leq h_n \leq 1$ for each $n$. By passing to a subsequence if necessary, we may assume that $(h_n)$ converges pointwise to a function $h$ because $\overline{\{ h_n(x)\mid n \in \mathbb{N} \}}$ is compact for each $x \in V$. By Lemma 3.2 of \cite{HoloSoar}, $h$ is $p$-harmonic on $V$. Since the sequence $(I_p(h_n, V))$ is bounded, Theorem 1.6 on page 177 of \cite{TaylorLay} says that by passing to a subsequence if necessary, we may assume that $(h_n)$ converges weakly to a function $\overline{h} \in D_p(\Gamma)$. Because evaluation by $x \in V$ is a continuous linear functional on $D_p(\Gamma)$, we have that $h_n(x) \rightarrow \overline{h}(x)$ for each $x \in V$. Thus $h = \overline{h}$ and $h \in BD_p(\Gamma)$. It follows from \cite[Corollary 4.9]{PulsPMJ} that $h = 0$ on $V$, due to that fact $h=0$ on $\partial_p(\Gamma)$.

Since $h_n \rightarrow h$ pointwise on $U_k$ for $k \in \mathbb{N}$, it follows $I_p(h_n, U_k) \rightarrow I_p(h, U_k) = 0$ for each $k$. Consequently, $I_p(h_n, V) \rightarrow 0$. By taking a subsequence if necessary, we may assume that $\parallel h_n \parallel_{D_p} \leq 2^{-n}$. Let $\epsilon > 0$ be given and for $m \in \mathbb{N}$, let $g_m = \sum_{k=1}^m h_k$. There exists $N \in \mathbb{N}$ such that $2^{-N} < \epsilon$. For $m, n \in \mathbb{N}$ with $m > n \geq N$ we see that 
\[ \parallel g_m - g_n \parallel_{D_p} = \parallel \sum_{k = n+1}^m  h_k \parallel_{D_p} \leq \sum_{k=n+1}^m 2^{-k} < 2^{-n} < \epsilon. \]
Hence, the Cauchy sequence $(g_m)$ converges to $g = \sum_{k=1}^{\infty} h_k$ in the $D_p$-norm. Thus $g \in D_p(\Gamma).$ For $x\in K, g_m(x) =m$, so $g(x) = \infty$; also $g = g_m =0$ on $\partial_p(\Gamma)$. The proof of the lemma is complete.
\end{proof}

The next result was shown to be true for the case $p = 2$ in \cite[Theorem 6.16]{Soardibook}. Our proof is essentially the same, and we include it for completeness. 
\begin{Lem}\label{infinitelength}
Let $P$ be a family of one-sided infinite paths in $\Gamma$ and let 
\[ K = \overline{\cup_{\gamma \in P} E(\gamma)}. \]
If $K$ is disjoint from $\partial_p(\Gamma)$, then $\lambda_p(P) = \infty$.
\end{Lem}
\begin{proof}
By Lemma \ref{gequalinfinity} there exists a $g \in D_p(\Gamma)$ such that $g = \infty$ on $K$ and $g = 0$ on $\partial_p(\Gamma)$. Let $\gamma \in P$ and let $x_1, x_2, x_3 \dots$ be the vertex representation of $\gamma$. Since $E(\gamma) \subseteq K$ we have that $g(\gamma) = \lim_{k \rightarrow \infty} g(x_k) = \infty$. Thus
\[ \sum_{k=1}^{\infty} \vert g(x_k) - g(x_{k+1}) \vert \geq \lim_{k \rightarrow \infty} (g(x_k) - g(x_1)) = \infty. \]
By \cite[Lemma 2.3]{KayYam} we obtain $\lambda_p(P) = \infty.$
\end{proof}

A connected infinite subset $D$ of $V$ with $\partial D \neq \emptyset$ is defined to be $D_p$-massive if there exists a $p$-harmonic function $u$ on $D$ that satisfies the following: $0 \leq u \leq 1, u =0 \mbox{ on } \partial D, \sup_D u =1$ and $I_p(u, D) < \infty$. The function $u$ is known as an inner potential of $D$.

\begin{Prop} \label{dpmassivefinitelength}
Let $D$ be a $D_p$-massive subset, with inner potential $u$, of $V$. Denote by $P_D$ the set of all one-sided infinite paths contained in $D \cup \partial D.$ Then $\lambda_p(P_D) < \infty.$
\end{Prop}
\begin{proof}
Let $a \in D$ and let $P_a$ be the set of all paths in $P_D$ with initial point $a$. If $\lambda_p(P_a) < \infty$, then $\lambda_p(P_D) < \infty$ by \cite[Lemma 2.1]{KayYam}. Let $(B_n)$ be an exhaustion of $V$ by finite connected subsets of $V$ such that $B_1 \cap \partial D \neq \emptyset$. Pick an $a \in B_1 \cap \partial D$. By combining Theorem 2.1 and Theorem 2.4 of \cite{NakamuraYamasaki} we see that $\lambda_p(P_a) < \infty$ if and only if $cap_p(\{a\}, \infty, D) > 0$. Thus to finish the proof we need to show $cap_p(\{a\}, \infty, D) > 0$, which we now proceed to do.

Choose admissible functions $\omega_k, k \geq 2$, for condensers $(\{a\}, (D \cup \partial D) \setminus B_k, D)$ such that 
\begin{equation}
I_p(\omega_k, D \cap B_k) \leq cap_p(\{a\}, (D \cup \partial D)\setminus B_k, D) + \frac{1}{k}. \label{eq:condenserineq}
\end{equation}
Replacing all values of $\omega_k(x)$ on $D \cap B_k$ for which $\omega_k (x) < 0$ by 0 and replacing all values of $\omega_k(x)$ on $D \cap B_k$ for which $\omega_k(x) > 1$ by 1 decreases the value of $I_p(\omega_k, D \cap B_k)$. Thus we may and do assume $0 \leq \omega_k \leq 1$ on $D \cap B_k$. Theorem 3.11 of \cite{HoloSoar} tells us that there exists a unique $p$-harmonic function $v_2$ on $D \cap B_2$ such that $v_2 = \omega_2$ on $\partial(D \cap B_2)$. Extend $v_2$ to all of $D$ by setting $v_2 = 1$ on $D \setminus B_2$. By the minimizing property of $p$-harmonic functions,
\[ I_p( v_2, D \cap B_2) \leq I_p(\omega_2, D \cap B_2). \]
Since $u$ is $p$-harmonic on $D$ and $u(x) \leq v_2(x)$ for all $x \in \partial(D \cap B_2), u \leq v_2$ on $D \cap B_2$ by \cite[Theorem 3.14]{HoloSoar}. Pick $\omega_3$. The set $A = \{ x \in D \mid \omega_3(x) > v_2(x) \}$ is a subset of $D \cap B_2$. If $ A \neq \emptyset$, redefine $\omega_3$ by setting $\omega_3 = v_2$ on $A$. The redefined $\omega_3$ decreases $I_p(\omega_3, D \cap B_3)$, so (\ref{eq:condenserineq}) remains true. By continuing as above, we obtain a decreasing sequence of functions $(v_k)$ such that $v_k$ is $p$-harmonic on $B \cap B_k, v_k \geq u$, and 
\[ I_p(v_k, D \cap B_k) \leq I_p(\omega_k, D \cap B_k). \]
Now assume that $cap_p(\{ a \}, (D \cup \partial D) \setminus B_k, D) \rightarrow 0$. Then $I_p( v_k, D \cap B_k) \rightarrow 0$. Since $v_k \geq u$ and $\sup_D u = 1$, it must be the case that $(v_k) \rightarrow 1_D$, the constant function 1 on $D$. This is a contradiction because $(v_k)$ is a decreasing sequence of functions, $0 \leq v_2 \leq 1$ and $v_2 \neq 1$. Thus, $cap_p(\{ a \}, \infty, D) > 0$ and the proof of the proposition is complete.
\end{proof}

Our next result is  \cite[Theorem 6.18]{Soardibook} for the case $p=2$. We give a different proof of the result.
\begin{Lem} \label{containedinclosure} 
Let $P$ be the family of all one-sided infinite paths in $\Gamma$ and let $P_{\infty} \subseteq P$ be any subfamily with $\lambda_p(P_{\infty}) = \infty.$ Then
\[  \partial_p(\Gamma) \subseteq \overline{\{ \cup_{\gamma} E(\gamma) \mid \gamma \in P \setminus P_{\infty} \}}. \]
\end{Lem}
\begin{proof}
Set $K = \overline{ \{ \cup_{\gamma} E(\gamma) \mid \gamma \in P \setminus P_{\infty} \} }.$ Since our standing assumption is that $\Gamma$ is $p$-hyperbolic, it follows from \cite[theorem 2.1]{NakamuraYamasaki} that $\lambda_p(P) < \infty$. By \cite[Lemma 2.2]{KayYam}, $\lambda_p(P \setminus P_{\infty}) < \infty$. Lemma \ref{infinitelength} tells us $K \cap \partial_p(\Gamma) \neq \emptyset.$ For purposes of contradiction, assume that there exists a $\chi \in \partial_p(\Gamma)$ for which $\chi \notin K$. By Urysohn's lemma there exists a continuous function $f$ on $Sp(BD_p(\Gamma))$ that satisfies the following: $0 \leq f \leq 1, f(\chi) =1$ and $f = 0$ on $K \cap \partial_p(\Gamma)$. By density of $BD_p(\Gamma)$ in $C(Sp(BD_p(\Gamma)))$ we assume $f \in BD_p(\Gamma)$. The $p$-Royden decomposition for $BD_p(\Gamma)$ yields a unique $p$-harmonic function $h$ on $V$ and a unique $g \in B(\overline{C_c(\Gamma)}_{D_p})$ such that $f = g + h$. Theorem 4.8 of \cite{PulsPMJ} shows that $g=0$ on $\partial_p(\Gamma)$. Combining this fact with the maximum principe (\cite[Theorem 4.7]{PulsPMJ}) it follows that $0 < h < 1$ on $V, h(\chi) = 1$ and $h=0$ on $\partial_p(\Gamma) \cap K$. Let
\[ A = \{ x \in V \mid h(x) > 1- \epsilon \}, \]
where $0 < \epsilon < 1$. Let $B$ be a component of $A$. The set $B$ is $D_p$-massive, see the proof of \cite[Proposition 4.12]{PulsPMJ} for a proof of this fact. Let $P_A$ be the family of all one-sided infinite paths in $A$, and let $P_B$ consist of all one-sided infinite paths in $B$. Since $B$ is a $D_p$-massive set, $\lambda_p(P_B) < \infty$ by Proposition \ref{dpmassivefinitelength}. It now follows from \cite[Lemma 2.1]{KayYam} that $\lambda_p(P_A) < \infty$. Set 
\[ K_1 = \overline{ \{ \cup_{\gamma} E(\gamma) \mid \gamma \in P_A \setminus P_{\infty} \} }. \]
Another appeal to Lemma \ref{infinitelength} shows $K_1 \cap \partial_p(\Gamma) \neq \emptyset$, because $\lambda_p(P_A \setminus P_{\infty}) < \infty$. Furthermore, $h = 0$ on $K_1 \cap \partial_p(\Gamma)$ since $K_1 \cap \partial_p(\Gamma) \subseteq K \cap \partial_p(\Gamma).$ However, $h(\gamma) \geq 1 - \epsilon$ for all $\gamma \in P_A$. Thus we obtain the contradiction $h(x) \geq 1 - \epsilon$ for all $x \in K_1$. Therefore, $\partial_p(\Gamma) \subseteq K$, as desired.
\end{proof}

{\em Proof of Theorem \ref{boundchar}}. Let $f \in B(\overline{C_c(\Gamma)}_{D_p})$ and let $a \in V$. Denote by $P_a$ the set of all one-sided infinite paths in $\Gamma$ with initial point $a$. Set
\[   A_{a,f} = \{ \gamma \in P_a \mid f(\gamma) \neq 0 \}. \]
By \cite[Theorem 3.3]{KayYam}, $\lambda_p(A_{a,f}) = \infty$. Also, \cite[Lemma 2.2]{KayYam} tells us $\lambda_p(A_f) = \lambda_p ( \cup_{a \in V} A_{a,f}) = \infty$. The definition of $A_f$ above and $E_f$ below were given in Section \ref{statementresults}. Now Proposition \ref{containedinclosure} says that 
\[ \partial_p(\Gamma) \subseteq E_f.\]
For notational convenience set $F = \cap_f E_f$, where $f$ runs through $B(\overline{C_c(\Gamma)}_{D_p})$. Thus, $\partial_p(\Gamma) \subseteq F$. We now proceed to prove the reverse inclusion. Suppose there exists a $\chi \in F$ for which $\chi \notin \partial_p(\Gamma)$. By \cite[Theorem 4.8]{PulsPMJ} we obtain an $f \in B(\overline{C_c(\Gamma)}_{D_p})$ for which $\chi(f) \neq 0$. Let $\alpha \sim x_0,x_1,\dots,x_n,\dots$ be a one-sided path with $\chi \in \overline{V}(\alpha)$. Because $\chi(f) \neq 0$, there is a subsequence $(x_{n_k})$ of $(x_n)$ that satisfies $\lim_{k \rightarrow \infty} f(x_{n_k}) \neq 0$. Thus $f(\alpha) \neq 0$ and has a result $\alpha \in A_f$. Hence $\chi \notin \{ \cup_{\gamma} E(\gamma) \mid \gamma \in P \setminus A_f \}$. We are assuming $\chi \in E_f$, so it must be the case that there is a sequence $(\chi_n)$ in $\{ \cup_{\gamma} E(\gamma) \mid \gamma \in P\setminus A_f \}$ with $(\chi_n) \rightarrow \chi$. Since $f(\gamma) = 0$ for each $\gamma \in P \setminus A_f$ it follows immediately that $\chi_n (f) = 0$ for each $n \in \mathbb{N}$. This implies $\chi(f) = 0$, contradicting our assumption $\chi(f) \neq 0$. Therefore, $F \subseteq \partial_p(\Gamma)$. The proof of Theorem \ref{boundchar} is now complete.

\bibliographystyle{plain}
\bibliography{charar_fsigma_pharmbound}

\begin{thebibliography}{1}

\bibitem{Gromov}
M.~Gromov.
\newblock Asymptotic invariants of infinite groups.
\newblock In {\em Geometric group theory, Vol.\ 2 (Sussex, 1991)}, volume 182
  of {\em London Math. Soc. Lecture Note Ser.}, pages 1--295. Cambridge Univ.
  Press, Cambridge, 1993.

\bibitem{HoloSoar}
Ilkka Holopainen and Paolo~M. Soardi.
\newblock {$p$}-harmonic functions on graphs and manifolds.
\newblock {\em Manuscripta Math.}, 94(1):95--110, 1997.

\bibitem{KayYam}
Takashi Kayano and Maretsugu Yamasaki.
\newblock Boundary limit of discrete {D}irichlet potentials.
\newblock {\em Hiroshima Math. J.}, 14(2):401--406, 1984.

\bibitem{NakamuraYamasaki}
Tadashi Nakamura and Maretsugu Yamasaki.
\newblock Generalized extremal length of an infinite network.
\newblock {\em Hiroshima Math. J.}, 6(1):95--111, 1976.

\bibitem{PulsPMJ}
Michael Puls.
\newblock Graphs of bounded degree and the {$p$}-harmonic boundary.
\newblock {\em Pacific J. Math.}, 248(2):429--452, 2010.

\bibitem{Puls06}
Michael~J. Puls.
\newblock The first {$L^p$}-cohomology of some finitely generated groups and
  {$p$}-harmonic functions.
\newblock {\em J. Funct. Anal.}, 237(2):391--401, 2006.

\bibitem{Soardibook}
Paolo~M. Soardi.
\newblock {\em Potential theory on infinite networks}, volume 1590 of {\em
  Lecture Notes in Mathematics}.
\newblock Springer-Verlag, Berlin, 1994.

\bibitem{TaylorLay}
Angus~E. Taylor and David~C. Lay.
\newblock {\em Introduction to functional analysis}.
\newblock Robert E. Krieger Publishing Co. Inc., Melbourne, FL, second edition,
  1986.

\bibitem{Yam76}
Maretsugu Yamasaki.
\newblock Parabolic and hyperbolic infinite networks.
\newblock {\em Hiroshima Math. J.}, 7(1):135--146, 1977.

\end{thebibliography}
\end{document}